\documentclass[12pt]{amsart}
\usepackage[top=1in, bottom=1in, left=1in, right=1in]{geometry}
\usepackage{times, amsthm, amssymb, amsmath, amsfonts, amsthm, bm,
  graphicx, mathrsfs} 
\usepackage{latexsym,wasysym}
\usepackage{amscd}
\usepackage{float}
\usepackage{enumerate}
\usepackage[all]{xy}
\usepackage[usenames,dvipsnames]{color}
\usepackage{tikz}
\usepackage[pagebackref=true]{hyperref}

 \numberwithin{equation}{section}

\usepackage{algpseudocode}

\usepackage{mathtools}

\usepackage{xy}
\input xy
\xyoption{all}

\newcommand{\excise}[1]{}

\newtheorem{theorem}{Theorem}[section]
\newtheorem{proposition}[theorem]{Proposition}
\newtheorem{lemma}[theorem]{Lemma}
\newtheorem{corollary}[theorem]{Corollary}

\theoremstyle{definition}
\newtheorem{definition}[theorem]{Definition}
\newtheorem{example}[theorem]{Example}
\newtheorem{remark}[theorem]{Remark}

\newtheorem{conv}[theorem]{Convention}

\numberwithin{equation}{section}

\newcommand\<{\langle}

\newcommand\NN{{\mathbb N}}
\newcommand\QQ{{\mathbb Q}}
\newcommand\RR{{\mathbb R}}

\newcommand\ZZ{{\mathbb Z}}

\newcommand\del{\partial}

\newcommand\rank{{\rm rank\,}}

\newcommand\Sat{{\rm Sat}}

\renewcommand\>{\rangle}

\def\endrk{\hfill$\hexagon$}

\newcommand*{\defeq}{\mathrel{\vcenter{\baselineskip0.5ex \lineskiplimit0pt
                     \hbox{\scriptsize.}\hbox{\scriptsize.}}}%
                     =}

\newcommand\kk{\Bbbk}
\newcommand\mm{{\mathfrak m}}

\newcommand\minus{\smallsetminus}

\def\ol#1{{\overline {#1}}}

\DeclareMathOperator\Ass{Ass} 
\DeclareMathOperator\sat{sat} 
\DeclareMathOperator\Hull{Hull} 
\DeclareMathOperator\supp{supp} 

\newcommand\toral{\textnormal{toral}}

\begin{document}

\mbox{}
\title{Some algebraic aspects of mesoprimary decomposition\qquad}
\author{Laura Felicia Matusevich}
\address{Mathematics Department\\Texas A\&M University\\College Station, TX 77843}
\email{laura@math.tamu.edu}
\author{Christopher O'Neill}
\address{Mathematics Department\\UC Davis\\Davis, CA 95616}
\email{coneill@math.ucdavis.edu}

\thanks{LFM was partially supported by NSF grant DMS-1500832.}
\subjclass[2010]{Primary: 05E40, 13A02. Secondary: 13F99, 13P99.}

\date{\today}

\begin{abstract}
\hspace{-2.05032pt}
Recent results of Kahle and Miller give a method of constructing 
primary decompositions of binomial ideals by first constructing 
``mesoprimary decompositions'' determined by their underlying 
monoid congruences.  Monoid congruences (and therefore, binomial ideals)
can present many subtle behaviors that must be 
carefully accounted for in order to produce general results,
and this makes the theory complicated.
In this paper, we examine their results in the presence of 
a positive $A$-grading, where certain pathologies are avoided 
and the theory becomes more accessible.  
Our approach is algebraic: while key notions for mesoprimary decomposition 
are developed first from a combinatorial point of view, here we
state definitions and results in algebraic terms, which are moreover
significantly simplified due to our (slightly) restricted setting. In
the case of toral components (which are well-behaved with respect to
the $A$-grading), we are able to obtain further simplifications under
additional assumptions.
We also provide counterexamples to two open questions, identifying 
(i) a binomial ideal whose hull is not binomial, 
answering a question of Eisenbud and Sturmfels, and 
(ii) a binomial ideal $I$ for which $I_\textnormal{toral}$ is not
binomial, answering a question of Dickenstein, Miller and the first
author.   
\end{abstract}
\maketitle

\section{Introduction}\label{s:intro}

A \emph{binomial} is a polynomial with at most two terms; a
\emph{binomial ideal} is an ideal generated by binomials. 
Monomial ideals, well known
as objects with rich combinatorial structure, are also
binomial. Toric ideals, also of much combinatorial interest, are
binomial as well.

Binomial ideals in general are known for being constrained in their
geometry and algebra: the irreducible components of a variety
defined by binomials
(over an algebraically closed field) are toric
varieties. More precisely, if the base field is algebraically closed,
the associated primes and primary components of binomial ideals are
binomial (and binomial prime ideals are isomorphic to toric ideals by
rescaling the variables). These results form the core of the
article~\cite{ES}.

The combinatorial study of binomial primary decomposition was started
in~\cite{dmm}, but the results in 
that article require the base field to be algebraically closed of
characteristic zero. 
To completely
eliminate assumptions on the base field, a new kind of decomposition,
called \emph{mesoprimary decomposition}, from which a primary
decomposition can be easily obtained, was introduced in~\cite{kmmeso}.
The main theme in~\cite{kmmeso} is that the combinatorial structures
underlying binomial ideals are \emph{monoid congruences}, that is, 
equivalence relations on a monoid that are compatible with the additive
structure. 

The starting point of~\cite{kmmeso} is that, when performing the
primary decomposition of a binomial ideal, the base field plays a role
only when one encounters lattice ideals 
(see Definition~\ref{def:latticeIdeal} and Remark~\ref{remark:BaseFieldLatticeIdeals} for details).
Indeed, one can decompose a binomial ideal into structurally simpler binomial
ideals without regard for the base field. For
instance,~\cite[Theorem~6.2]{ES} provides a base field independent
decomposition of a binomial ideal into cellular binomial ideals (Definition~\ref{def:cellular})
From there, one can proceed, again
without field assumptions, to more refined \emph{unmixed decompositions}
(see~\cite[Corollary~8.2]{ES} for characteristic zero
and~\cite[Section~4]{OS}, \cite[Theorem~5.1]{EM} for results
without field assumptions). From unmixed decompositions, one can
fairly explicitly obtain primary decompositions. However, unmixed decompositions
are not the most refined possible binomial decompositions, nor do they
completely reveal the combinatorial structure of the underlying
binomial ideal; the mesoprimary decompositions of~\cite{kmmeso} fulfil those goals.

Monoid congruences (and therefore, binomial ideals)
can present many subtle behaviors. To produce general results, these
must be carefully accounted for, and this makes the theory complicated.
In order to simplify the definitions and results on monoid congruences
required to perform mesoprimary decomposition, 
we restrict our attention in this article to the important class
of \emph{positively graded binomial ideals}.
Under this assumption, certain pathologies
for the corresponding congruences are avoided, and the theory becomes
more accessible.  Our approach is algebraic: while
in~\cite{kmmeso}, key notions are developed first from a combinatorial
point of view, here we
restate definitions and results in algebraic terms, which are moreover
significantly simplified due to our (slightly) restricted setting.
These definitions can be found in Section~\ref{s:mesodecomp}, 
after we review the necessary background on binomial ideals 
from \cite{ES} in Section~\ref{s:background}.  

The remainder of the paper concerns results and ideas 
from \cite{dmm,dmm2} that identify, in~the $A$-homogeneous setting, 
certain primary components (called \emph{toral} components)
that inherit sufficient combinatorial structure from the grading 
to make them easier to compute.  
One of the main goals of this project was to obtain 
analogous methods for computing \emph{toral} mesoprimary components 
(Definition~\ref{d:toral}).  
Much to our surprise, the combinatorial methods explored 
in~\cite{dmm} and~\cite{kmmeso} appear to be somehow incompatible;
each utilizes some underlying combinatorial structure
to simplify computation of primary decomposition, 
but in sufficiently different ways that it is difficult 
to simultaneously benefit from both outside of highly restricted cases.  

Sections~\ref{sec:setTo1} and~\ref{sec:Itoral} 
contain some mesoprimary analogs of results from \cite{dmm} 
in special cases, 
as well as examples demonstrating why more general 
results in this direction are difficult to obtain.  
We also idenfity in Example~\ref{e:nonbinomialItoral} 
a binomial ideal $I$ with the property that 
the intersection of the toral primary components of $I$ 
is not a binomial ideal,
thus answering a question posed by the authors of \cite{dmm}.  

We close this section by addressing a question of Eisenbud and
Sturmfels. We recall that the \emph{Hull} of an ideal $I$, denoted
$\Hull(I)$ is the intersection of all the minimal primary
components of $I$. Corollary~6.5 of~\cite{ES} (see
also~\cite[Theorem~2.10]{EM}) states that the Hull of 
a cellular binomial ideal is binomial.

Problem~6.6 in~\cite{ES} asks: 
\begin{quote}
Is $\Hull(I)$ binomial for every (not necessarily cellular) binomial
ideal $I$?
\end{quote}
We provide a negative answer to this question in the following example.

\begin{example}
\label{e:nonbinomialhull}
The binomial ideal
$$\begin{array}{rcl}
I 
&=& \<x_4^2 - 1, x_1^2(x_4 - 1), x_3(x_4 - 1), x_3(x_1 - x_2), x_1^2 - x_1x_2, x_1x_2 - x_2^2, x_1^3, x_1^2x_3\> \\
&=& \<x_4 + 1, x_1^2, x_1x_2, x_2^2, x_3\> \cap \<x_4 - 1, x_1 - x_2, x_1^2\> \cap \<x_4 - 1, x_1^2 - x_1x_2, x_1x_2 - x_2^2, x_1^3, x_3\>
\end{array}$$
appeared in \cite[Example~16.10]{kmmeso}. It has three associated
primes
\[
\< x_1,x_2,x_3,x_4+1\> \text{ (minimal)}, \; \<x_1,x_2,x_4-1\> \text{
  (minimal)}, \; \text{ and } \<x_1,x_2,x_3, x_4-1\> \text{ (embedded)}.
\]

The intersection of the minimal primary components of $I$ is 
$$\Hull(I) = \<x_4^2 - 1, x_3x_4 - x_3, x_1x_4 - x_2x_4 + x_1 - x_2,
x_1x_3 - x_2x_3, x_2^2, x_1x_2, x_1^2\>,$$
whose non-binomiality can be verified with \texttt{Macaulay2} 
package \texttt{Binomials}, or by simply computing a reduced Gr\"obner basis.  
\endrk
\end{example}

\subsection*{Acknowledgements}
We are grateful to Alicia Dickenstein, Thomas Kahle and Ezra Miller
for many inspiring conversations; we also thank Ezra Miller, 
Ignacio Ojeda, and two anonymous referees for their useful suggestions 
on a previous version of this article.

\section{Binomial primary decomposition}\label{s:background}

In this section, we give background on binomial ideals and recall terminology from \cite{ES}.  

\subsection*{Conventions}

Throughout this article we use $\NN = \{0,1,2,\dots\}$.  
Unless otherwise stated, $\kk$ denotes an arbitrary field. We denote
$\kk[\NN^n]$ the polynomial ring in $n$ variables $x_1,\dots,x_n$.  Given
$\sigma \subset [n] = \{1,2,\ldots,n\}$, write $\kk[\NN^\sigma]$ for
the subring of $\kk[\NN^n]$ generated by $x_i$ for $i \in \sigma$, and
$\mm_\sigma = \<x_i \mid i \in \sigma\>$ for the maximal monomial
ideal in $\kk[\NN^\sigma]$.  If $\sigma \subset [n]$, we denote $\sigma^c = [n] \minus \sigma$.

Throughout this article, a \emph{lattice} is a finitely generated free
abelian group, and a \emph{character} on a lattice $L$ is a group
homomorphism $\rho:L\to \kk^*$.

\subsection*{Cellular ideals, lattice ideals, primary decomposition}

We first introduce an important class of binomial ideals.

\begin{definition}
\label{def:cellular}
A binomial ideal $I\subset \kk[\NN^n]$ is \emph{cellular} if each $x_i$ is
either nilpotent or a nonzerodivisor modulo $I$ for $i=1,\dots,n$.
If $I$ is cellular, the nonzerodivisor variables modulo $I$ are called the
\emph{cellular variables} modulo $I$.
If  $\sigma \subseteq [n]$ indexes 
the cellular variables modulo $I$, then
$I$ is \emph{$\sigma$-cellular}.
\end{definition}

The following result underscores the importance of cellular binomial ideals.

\begin{theorem}{\cite[Theorem~6.2]{ES}}
\label{thm:binomialCellularDec}
Every binomial ideal in $\kk[\NN^n]$ can be expressed as a finite intersection of cellular binomial ideals.
\end{theorem}

Another important class of binomial ideals are lattice ideals.

\begin{definition}
\label{def:latticeIdeal}
Let $L \subseteq \ZZ^n$ be a lattice and $\rho:L\to \kk^*$ a character
on $L$. The \emph{lattice ideal corresponding to $L$ and $\rho$} is:
\[
I(\rho) \defeq \< x^u - \rho(u-v) x^v \mid u,v \in \NN^n, u-v \in L \>.
\]
We omit $L$ from the notation for a lattice ideal, since it is
understood that $L$ is specified when the character $\rho$ is given.
An ideal $I$ is called a \emph{lattice ideal} if there exists a
character $\rho$ on a lattice $L$ such that $I=I(\rho)$. 
\end{definition}

The following result is a characterization of lattice ideals in terms
of cellular binomial ideals.

\begin{lemma}{\cite[Corollary~2.5]{ES}}
Let $I\subset \kk[\NN^n]$ be a binomial ideal. Then $I$ is
$[n]$-cellular if and only if $I$ is a lattice ideal.
\end{lemma}

If $L$ is a lattice in $\ZZ^n$, we define its \emph{saturation} to be
$\Sat(L) \defeq (\QQ\otimes_{\ZZ}L) \cap \ZZ^n$; $L$ is \emph{saturated} if $L = \Sat(L)$. Saturated lattices correspond to prime
binomial ideals when $\kk$ is algebraically closed, as the following results show.

\begin{lemma}{\cite[Theorem~2.1]{ES}}
Assume $\kk$ is algebraically closed.
A lattice ideal in $\kk[\NN^n]$ is prime if and only if the underlying lattice is
saturated.
\end{lemma}

\begin{theorem}{\cite[Corollary~2.6]{ES}}
Assume $\kk$ is algebraically closed.
A binomial ideal in $\kk[\NN^n]$ is prime if and only if it is of the
form $\mm_{\sigma^c} + \kk[\NN^n] I$ where $I$ is a prime lattice
ideal in $\kk[\NN^\sigma]$. From~now on, where it causes no confusion, we suppress $\kk[\NN^n]$ and
use $\mm_{\sigma^c} + I$ instead. 
\end{theorem}

The primary decomposition of lattice ideals can be done explicitly,
when the base field is algebraically closed. Stating this result is
our next task.

Let $L$ be a lattice in $\ZZ^n$, let $\rho:L\to \kk^*$ a character
on $L$, and let $p$ be a prime number.
We define $\Sat_{p}(L)$ and $\Sat'_{p}(L)$ to be the largest
sublattices of $\Sat(L)$ containing $L$ so that
$|\Sat_p(L)/L | = p^k$ for some $k \in \ZZ$, and
$| \Sat'_p(L)/L | = g$ where $(p, g) = 1$. We also adopt the
convention that $\Sat_0(L) \defeq L$ and $\Sat'_0(L)
\defeq \Sat(L)$. 

\begin{theorem}[{\cite[Corollary~2.2]{ES}}]
\label{thm:latticePrimaryDecomposition}
Let $\kk$ be algebraically closed field of characteristic $p \geq
0$, and let $L, \rho$ as above.
There are $g = | \Sat'_p(L)/L|$ distinct characters
$\rho_1,\dots,\rho_g :\Sat'_p(L) \to \kk^*$ that extend $\rho$. For each
$i=1,\dots,g$, there exists a unique partial character $\chi_{i}$
that extends $\rho_{i}$ to  $\Sat(L_{\rho})$. (If $p=0$, $\chi_i =
\rho_i$ for all $i=1,\dots,g$.)
The associated primes of the lattice ideal $I(\rho)$ are
$I(\chi_1),\dots,I(\chi_g)$, they are all minimal, and have the same
codimension $\rank(L)$. For 
each $i=1,\dots,g$, the ideal $I(\rho_i)$ is $I(\chi_i)$-primary, and 
\[
I(\rho) = \bigcap_{i=1}^g I(\rho_i)
\]
is the minimal primary decomposition of $I(\rho)$. 
\qed
\end{theorem}

We finally state the main result of~\cite{ES}.

\begin{theorem}{\cite[Theorem~7.1]{ES}}
\label{thm:binomialPrimDec}
Let $\kk$ be an algebraically closed field. Every binomial ideal $I$ in
$\kk[\NN^n]$ has a minimal primary decomposition in terms of binomial
ideals; in other words, the associated primes of $I$ are binomial, and its
primary components can be chosen binomial.
\end{theorem}

\begin{remark}
\label{remark:BaseFieldLatticeIdeals}
The assumption in Theorem~\ref{thm:binomialPrimDec} that $\kk$ is algebraically closed is necessary.
This can be seen at
the level of lattice ideals
(Theorem~\ref{thm:latticePrimaryDecomposition}), even in one variable
(consider the ideal $\<y^p-1\> \subseteq \kk[y]$). 
Also, it is clear from the aforementioned examples
 that the characteristic of the
base field makes a difference 
in the primary decomposition of binomial ideals.
However, it is only when decomposing lattice ideals that base field
considerations enter. Before this stage, for instance, when performing
cellular decompositions, the base field does not play a role.
\endrk
\end{remark}

\section{Mesoprimary decomposition of positively graded binomial ideals}\label{s:mesodecomp}

In this section we (re)define terms from \cite{kmmeso} 
in the language of commutative algebra rather than monoid
congruences, and provide several examples.  

\newpage

\begin{definition}
\label{d:mesoprime-mesoprimary} 
\leavevmode
\begin{enumerate}[\label=(a)]
\item
\label{d:mesoprime}
A \emph{mesoprime ideal} is an ideal of the form
$(\kk[\NN^n])I_\text{lat} + \mm_{\sigma^c}$, where
$\sigma \subset [n]$, and $I_\text{lat} \subset \kk[\NN^\sigma]$ is a
lattice ideal.
\item
\label{d:mesoprimary}
An ideal $I$ is \emph{mesoprimary} if it is $\sigma$-cellular and 
$(I : x^m) \cap \kk[\NN^\sigma] = I \cap \kk[\NN^\sigma]$ for every
monomial $x^m \notin I$. In this case, the \emph{associated mesoprime}
of $I$ is $\kk[\NN^n] (I \cap \kk[\NN^\sigma]) + \mm_{\sigma^c}$.
\end{enumerate}
\end{definition}

\begin{remark}\label{r:mesoprime-mesoprimary}
Definition~\ref{d:mesoprime-mesoprimary}.\ref{d:mesoprime} 
is immediately equivalent to \cite[Definition~10.4.4]{kmmeso}, 
as is Definition~\ref{d:mesoprime-mesoprimary}.\ref{d:mesoprimary} 
to~\cite[Definition~10.4.2]{kmmeso}.  Indeed, $\sigma$-cellular 
ideals induce primary congruences, and the second condition 
in Definition~\ref{d:mesoprime-mesoprimary}.\ref{d:mesoprimary} 
ensures that $I$ is maximal among binomial ideals inducing the same 
congruence \cite[Remark~10.5, Corollary~6.7]{kmmeso}.  

Definition~\ref{d:mesoprime-mesoprimary}.\ref{d:mesoprimary} can be
equivalently stated in a way reminiscent of the definition of primary
ideals. Indeed, a binomial ideal $I$ is mesoprimary
to $I_\textnormal{meso} = I_\textnormal{lat} + \<x_i \mid
i \notin \sigma\>$ for some lattice ideal $I_\textnormal{lat}$ if~and
only if $I$ is $\sigma$-cellular and whenever $mb \in I$, where $m$ is a
monomial in the $\sigma^c$-variables and~$b$ is a binomial in the
$\sigma$-variables, then either $m \in I$ or $b \in I_\textnormal{lat}$. 
\endrk
\end{remark}

\begin{remark}
While we have expressed the definition of a mesoprimary ideal 
in algebraic terms, this is a combinatorial condition.
For instance, note that while
$I=\<x^3-1,y(x-1),y^3\> \subset \kk[x,y]$ is $\{x\}$-cellular and
$\Ass(I) = \Ass(\<x^3-1,y\>)$ (the latter ideal being a mesoprime),
$I$ itself is not mesoprimary, as $(I:y) \cap \kk[x] = \<x-1\> \neq
\<x^3-1\> = I \cap \kk[x]$.

We also remark that an ideal of the form $\kk[\NN^n]
I_{\textnormal{lat}} + \kk[\NN^n] I_{\textnormal{art}}$, 
where $I_{\textnormal{lat}}$ is a lattice ideal in 
$\kk[\NN^\sigma]$ and $I_{\textnormal{art}}$ is an
artinian binomial ideal in $\kk[\NN^{\sigma^c}]$ is always
mesoprimary, but not all mesoprimary ideals are of this form.  
For example, consider $I = \<x^2y^2-1, xz-yw, z^2, zw, w^2\> \subset
\kk[x,y,z,w]$, which is $\{x,y\}$-cellular and the intersection 
$I \cap \kk[x,y]$ is the lattice ideal generated by $x^2y^2-1$. 
If $I$ were of the form $I_{\textnormal{lat}}+I_{\textnormal{art}}$, 
then neither the lattice part nor the artinian part could account 
for the binomial $xz-yw$.
On the other hand, this ideal is mesoprimary, as can be checked 
by computing the ideal quotients with $z$, $w$, and $zw$, 
whose intersection with $\kk[x,y]$ equals $\<x^2y^2-1\>$.
\endrk
\end{remark}

Mesoprimary ideals are easy to primarily decompose, as doing so requires simply computing a primary decomposition for the underlying lattice ideal.  

\begin{proposition}[{\cite[Corollary~15.2 and~Proposition~15.4]{kmmeso}}]
\label{p:mesoprimaryPrimDec}
Let $I$ be a ($\sigma$-cellular) mesoprimary ideal, and denote by
$I_{\textnormal{lat}}$ the lattice ideal $I \cap
\kk[\NN^\sigma]$. 
The associated primes of $I$ are exactly the (minimal) primes of its
associated mesoprime $\kk[\NN^\sigma] I_{\textnormal{lat}}+
\mm_{\sigma^c}$.
Assume that $\kk$ is algebraically closed, and let
$I_{\textnormal{lat}} = \cap_{j=1}^g I_j$ be the primary decomposition
of $I_{\textnormal{lat}}$ from
Theorem~\ref{thm:latticePrimaryDecomposition}.
Then
\[
I = \cap_{j=1}^g (I+I_j)
\]
is the (canonical) primary decomposition of $I$.
\end{proposition}

Arguably the most important objects in~\cite{kmmeso} are the witnesses
(Definition~\ref{d:witness}), which are used as a starting place for
constructing the mesoprimary components in
Theorem~\ref{t:kmessential}. 

\cite[Definition~12.1]{kmmeso}, which introduces witnesses,
is complicated due to the 
need to account for the many pathologies that binomial ideals may
present. In this article we evade some of these pathologies
(and significantly simplify the definition of witnesses as a consequence) by
assuming that our binomial ideals are graded with respect 
to a positive grading (Definition~\ref{def:Agraded}).

\begin{definition}
\label{def:Agraded}
Let $A$ be a $d\times n$ integer matrix of full rank $d$, whose
columns span $\ZZ^d$ as a lattice. The matrix $A$ induces a
$\ZZ^d$-grading on $\kk[\NN^n]$, called the
\emph{$A$-grading}, by setting the degree of $x_i$ to be the $i$th
column of $A$. 
If the columns of $A$ belong to an open half-space defined by a
hyperplane through the origin, then the $A$-grading is \emph{positive}.
This condition implies that the cone consisting of the nonnegative real
combinations of the columns of $A$ is \emph{pointed} or 
\emph{strongly convex} (meaning that it contains no lines)
and no $x_i$ has degree zero.
\end{definition}

\begin{conv}
From now on, any $A$-grading on a binomial ideal is assumed to be positive.
\end{conv}

The standard $\ZZ$-grading is a positive $A$-grading, where $A$ is the
$1\times n$ matrix all of whose entries are ones. 
Note that when $\kk[\NN^n]$ is positively graded, the only monomial of
degree $0$ is $x^0=1$.
Additionally, an
$A$-grading is positive if and only if the columns of $A$ span $\ZZ^d$
as a lattice and there exists a $1\times d$
matrix $h$ such that all of the entries of $hA$ are strictly
positive. 
This implies that there cannot be divisibility relations
among monomials of the same degree. Indeed, if $u,v \in \NN^n$, $u
\neq v$ and $x^u$ divides $x^v$, then $hAu < hAv$, which implies that
$Au\neq Av$.

\begin{definition}
\label{d:witness}
Fix $\sigma \subset [n]$ and an
$A$-homogeneous binomial ideal $I \subset \kk[\NN^n]$, and set
$$I_\sigma=\bigg(I:\bigg(\prod_{i\in \sigma}x_i\bigg)^\infty\bigg).$$
\begin{enumerate}[(a)]
\item 
A \emph{monomial $I$-witness for
$\mm_{\sigma^c}$} is a monomial $x^w \in \kk[\NN^{\sigma^c}]$, $x^w
\notin I_\sigma$, such that there exists
$x^m \in \kk[\NN^\sigma]$ with the property that for each $i \notin
\sigma$, there are
a monomial $x^{q_i}$ and a scalar $\lambda_i \in \kk^*$ such
that $x_i (x^m x^w - \lambda_i x^{q_i}) \in I_\sigma$ but $x^mx^w
-\lambda_i x^{q_i} \notin I_\sigma$.  
\item 
A monomial $I$-witness $x^w$ for $\mm_{\sigma^c}$ is
\emph{essential} if there exist
an $A$-homogeneous
polynomial $p \in \kk[\NN^n]$,
$p\notin I_\sigma$, and a
monomial $x^v \in \kk[\NN^\sigma]$ such that $x^vx^w$ is a monomial in
$p$, and $x_j p \in I_\sigma$ for all $j \notin \sigma$.
\end{enumerate}
\end{definition}

Note that it is acceptable to take $\sigma = [n]$ in the above
definition, so that $\mm_{\sigma^c} = \< 0 \>$. We see that $x^w = 1$ 
is an essential monomial $I$-witness for $\<0\>$,
as most requirements of the definition do not apply in this case.

\begin{proposition}
\label{p:witness}
Fix $\sigma \subset [n]$,
an $A$-homogeneous binomial ideal $I \subset \kk[\NN^n]$, and
$x^w\in \kk[\NN^{\sigma^c}]$ with 
$x^w \notin I$.  Then $x^w$ is an (essential) monomial $I$-witness for
$\mm_{\sigma^c}$ in the sense of Definition~\ref{d:witness} if
and only if $I$ is an (essential) monomial $I$-witness for
$\mm_{\sigma^c}$ in the sense of \cite[Definition~12.1]{kmmeso}.   
\end{proposition}

\begin{proof}
Upon examining the prerequisite definitions for monomial $I$-witnesses
\cite[Definition~12.1]{kmmeso}, the only difference between that
statement and ours is that any monomial $I$-witness $x^w$ cannot be 
exclusively maximal \cite[Definition~4.7]{kmmeso}.  
In particular, we must show that for each $i \notin \sigma$, 
some choice of $x^{q_i}$ in Definition~\ref{d:witness} 
does not divide $x^w$.  

If $\sigma=[n]$, there is nothing to check, so assume that $\sigma
\subsetneq [n]$. Recall that 
$I_\sigma = \big( I: (\prod_{i\in\sigma}x_i)^\infty \big)$.
Since $I$ is $A$-homogeneous, so is $I_\sigma$. As $x^w \notin
I_\sigma$, and therefore, $x^m x^w \notin I_\sigma$, this implies that
each monomial $x^{q_i}$ satisfying $x_i(x^mx^w - \lambda_i
x^{q_i}) \in I_\sigma$ must have the same $A$-degree as $x^mx^w$. Given
that the $A$-grading is positive, we conclude that
the set $\{x^m x^w,x^{q_i}\}$ has no divisibility relations, so $x^w$ is
not exclusively maximal. 

Lastly, upon comparing the definition of esssential monomial $I$-witness
to \cite[Definition~12.1]{kmmeso}, the only difference is that
the latter requires the monomial $x^vx^w$ not be divisible by any other terms of $p$.  
This follows immediately from the fact that $p$ is $A$-homogeneous, 
as this implies that there are no divisibility relations among its nonzero monomials.  
\end{proof}

\begin{definition}
\label{d:mesoprimarycomponents}
Fix $\sigma \subset [n]$, a binomial ideal $I \subset \kk[\NN^n]$, and a monomial $x^m \notin I$.  Let 
$$I_{m}^\sigma \defeq ((I : (\textstyle\prod_{i\in \sigma} x_i)^\infty):x^m) \cap \kk[\NN^\sigma].$$
\begin{enumerate}[(a)]
\item The \emph{mesoprime at $x^m$} is the mesoprime ideal
$I_m^\sigma + \<x_i \mid i \notin \sigma\>$.
\item
A mesoprime is \emph{associated to $I$} if it equals the
mesoprime at an essential monomial $I$-witness.  
\item
The \emph{coprincipal component cogenerated by $x^m$} is the binomial ideal 
$$W_{x^m}^\sigma(I) = ((I + I_m^\sigma) : (\textstyle\prod_{i \in \sigma}x_i)^\infty) + M_{x^m}^\sigma(I),$$
where $M_{x^m}^\sigma(I) \subset \kk[\NN^n]$ is the ideal generated by
monomials $x^u \in \kk[\NN^n]$ such that 
$$x^m \notin ((I+\<x^u\>):(\textstyle\prod_{i\in \sigma} x_i)^\infty).$$
In general, we say that a monomial ideal $I$ is \emph{cogenerated by a set
  of monomials $M$} if the set of monomials not in $I$ consists of all
the monomials divisible by at least one element of $M$.  
\end{enumerate}
\end{definition}

Note that as a direct consequence of \cite[Proposition~12.17]{kmmeso}, coprincipal components cogenerated by essential witnesses are mesoprimary.  

\begin{remark}
\label{r:mesoprimarycomponents}
The equivalence of Definition~\ref{d:mesoprimarycomponents} to those
in \cite[Section~12]{kmmeso} follows upon unraveling prerequisite
definitions.  In particular, resuming notation from
Definition~\ref{d:witness}, the ideal $M_{x^m}^\sigma(I)$ contains the
same monomials as the ideal in \cite[Definition~12.13]{kmmeso} since 
$$x^m \notin \<x^u\> \subset \kk[\NN^n][x_i^{-1} \mid i \in \sigma]/I$$
precisely when $x^m$ is nonzero in $\kk[\NN^n][x_i^{-1} \mid i \in \sigma]/(I+\<x^u\>)$.  
\endrk
\end{remark}

\begin{theorem}
[{\cite[Theorem~13.3]{kmmeso}}]\label{t:kmessential}
Every binomial ideal $I \subset \kk[\NN^n]$ is the intersection of the
coprincipal components cogenerated by its essential witnesses.   
\end{theorem}

\begin{example}\label{e:mesoprimarycomponents}
Let $I = \<x^2 - y^2, x^2y - xy^2\> \subset \kk[x,y]$.  The ideal $I$
has two distinct mesoprimary components whose associated mesoprime is
the maximal monomial ideal $\<x,y\>$, each of which is cogenerated by
a witness monomial of total degree 2.  The monomial witnesses are $xy$,
and $x^2$, $y^2$. The latter two are considered as a single monomial $I$-witness,
since they are equal modulo $I$. 
The full mesoprimary
decomposition of $I$ produced by Theorem~\ref{t:kmessential} is given
by 
$$I = \<x - y\> \cap \<x^2, y^2\> \cap \<x^2 - y^2, x^3, xy, y^3\>,$$
and Figure~\ref{f:mesoprimarycomponents} depicts the binomial elements of
$I$ and the latter two mesoprimary components.  
\endrk
\end{example}

\begin{figure}
\begin{center}
\includegraphics[width=1.5in]{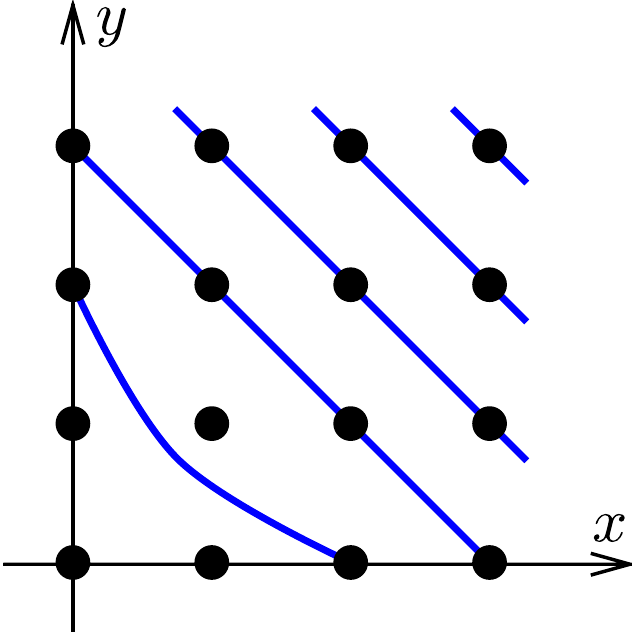}
\hspace{0.3in}
\includegraphics[width=1.5in]{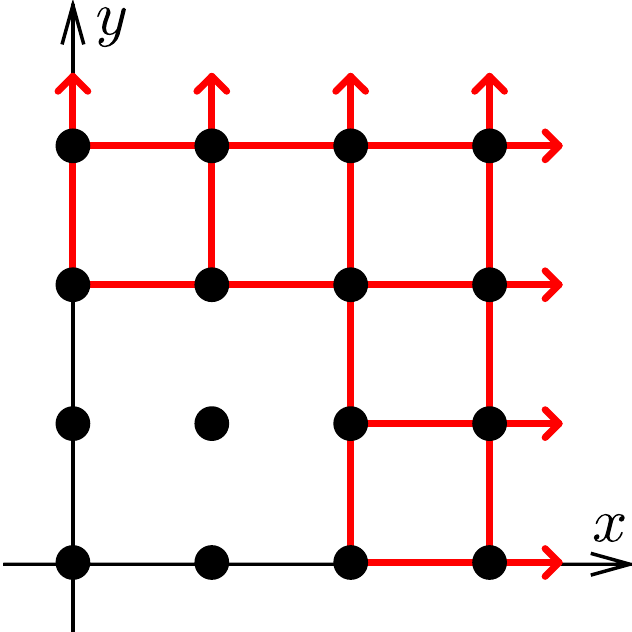}
\hspace{0.3in}
\includegraphics[width=1.5in]{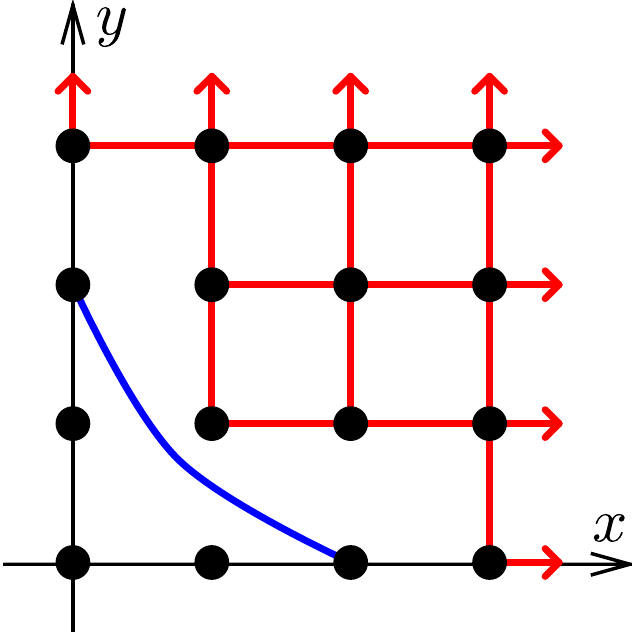}
\end{center}
\caption{Each line segment above represents a binomial element 
of the ideal~$I$ in Example~\ref{e:mesoprimarycomponents} (left) 
and its mesoprimary components whose associated mesoprime is 
the maximal monomial ideal (middle and right).  In particular, 
each connected component represents a monomial equivalence class 
modulo the given ideal.}
\label{f:mesoprimarycomponents}
\end{figure}

\begin{remark}\label{r:computation}
The difficulty in computing the coprincipal components of an ideal $I$
in Theorem~\ref{t:kmessential} is in locating the essential witnesses
of $I$.  Indeed, once a witness $x^m$ is known, computing the
coprincipal component $W_{x^m}^\sigma(I)$ amounts to computing a
saturation and the monomial ideal $M_{x^m}^\sigma(I)$, which is simply
the intersection of the irreducible monomial ideals whose quotients
have a maximal nonzero monomial of the form $x^{m'}$ with $x^{m'} -
\lambda x^m \in I$ for some $\lambda \in \kk$.   
\endrk
\end{remark}

Theorem~\ref{t:kmessential} and Proposition~\ref{p:mesoprimaryPrimDec}
produce a primary decomposition of any binomial ideal, and make 
no assumptions on the field $\kk$.   It is also possible to produce 
an irreducible decomposition using the underlying monoid congruence; 
see~\cite{kmo} for details on this construction.  

\begin{theorem}[{\cite[Theorems~15.6 and 15.11]{kmmeso}}]
Fix a binomial ideal $I \subseteq \kk[Q]$.  Each associated prime
of~$I$ is minimal over some associated mesoprime of~$I$.  If\/~$\kk =
\ol\kk$ is algebraically closed, then refining any mesoprimary
decomposition of~$I$ by canonical primary decomposition of its
components yields a binomial primary decomposition of~$I$.
\end{theorem}

\section{Toral and Andean mesoprimary components}
\label{s:toral}

As we have seen before, the assumption that a binomial ideal $I$ is
$A$-homogeneous carries with it a significant simplification of the
definition of witness from~\cite{kmmeso}. 
In general, the primary components of an $A$-homogeneous ideal are
$A$-homogeneous. If $I$ is $A$-homogeneous, then the coprincipal components
from~\ref{d:mesoprimarycomponents} are $A$-homogeneous as well, since
taking colon with monomials preserves the grading. Thus, any $A$-homogeneous
binomial ideal has an $A$-homogeneous mesoprimary decomposition
by Theorem~\ref{t:kmessential}.

Among all $A$-homogeneous binomial prime ideals, the \emph{toric ideal
  $I_A$} (the lattice ideal corresponding to the saturated lattice
$\ker_\ZZ(A)$ and the trivial character) is of particular interest. An
important property of this ideal is that it is \emph{finely graded},
meaning that the $A$-graded Hilbert function of $\kk[\NN^n]/I_A$ is
either $0$ or $1$.
It was noted in~\cite{dmm,dmm2} that when primary decomposing an
$A$-homogeneous binomial ideal, components corresponding to associated primes which are ``close'' to finely
graded are easier to compute (\cite[Theorem~4.13]{dmm2}). This
behavior subdivides the $A$-homogeneous binomial primes into two classes (Definition~\ref{d:toral}), namely
\emph{toral} (close to toric ideals) and \emph{Andean} (see
Remark~\ref{r:toral}), according to the behavior of their $A$-homogeneous
Hilbert function. 
In this section, we examine the $A$-graded Hilbert functions of
mesoprimes and mesoprimary ideals in the same spirit.

For $\sigma \subset [n]$, denote by $A_\sigma$ the matrix consisting
of the columns of $A$ indexed by $\sigma$.  

\begin{lemma}
\label{l:toralConditions}
For an $A$-homogeneous ($\sigma$-cellular) mesoprimary ideal $I \subset \kk[\NN^n]$, the following are equivalent.  
\begin{enumerate}[(a)]
\item \label{l:toralConditionsHilbertFunction}
The $A$-graded Hilbert function of $\kk[\NN^n]/I$ is bounded above.
\item \label{l:toralConditionsLatticeKernel}
If $L\subset \ZZ^\sigma$ is the lattice underlying
$I_{\textnormal{lat}} = I \cap \kk[\NN^{\sigma}]$, then $\sat(L)=\ker_\ZZ (A_\sigma)$.
\item \label{l:toralConditionsDimensionRank}
$\dim(\kk[\NN^n]/I) = \rank(A_\sigma)$.
\end{enumerate}
\end{lemma}

\begin{proof}
Since passing to an algebraic closure of $\kk$ changes neither the
$A$-graded Hilbert function nor the dimension of $\kk[\NN^n]/I$, we
assume for convenience that $\kk$ is algebraically closed.
By Proposition~\ref{p:mesoprimaryPrimDec}, if 
$I_{\textnormal{lat}}$ has primary decomposition $\cap_{j=1}^g I_j$, where $I_j$ are lattice
ideals whose underlying lattice is $\Sat(L)$, then $I = \cap_{j=1}^g
(I+I_j)$ is the (binomial) primary decomposition of $I$. By 
Theorem~\ref{thm:latticePrimaryDecomposition}, 
$\dim(\kk[\NN^\sigma]/I_j) = |\sigma|-\rank(L)$.

We first consider the case that $L$ is saturated, so that $I$ is
primary to (the prime ideal) $I_{\textnormal{lat}}+\mm_{\sigma^c}$. In this case,
proceeding as in~\cite[Example~4.6]{dmm}, $\kk[\NN^n]/I$ has a finite
filtration whose successive quotients are torsion free modules of rank
$1$ over the affine semigroup ring
$\kk[\NN^n]/(I_{\textnormal{lat}}+\mm_{\sigma^c})$. By induction on
the length of this filtration we reduce the proof to the case when $I$
is prime, in which case all the above conditions are clearly equivalent.

When $I$ is not necessarily primary,
the $A$-homogeneous maps
\[
\kk[\NN^n]/I \twoheadrightarrow \kk[\NN^n]/(I+I_j)
\qquad \textnormal{and} \qquad
\kk[\NN^n]/I \hookrightarrow \bigoplus_{j=1}^g \kk[\NN^n]/(I+I_j)
\]
imply that the $A$-graded Hilbert function of $\kk[\NN^n]/I$ is
bounded below by the $A$-graded Hilbert function of 
$\kk[\NN^n]/(I+I_1)$ and bounded above by the sum
of the $A$-graded Hilbert functions of
$\kk[\NN^n]/(I+I_j)$ for $j=1,\dots,g$. 

Note that the $A$-graded Hilbert functions of the rings
$\kk[\NN^n]/(I+I_j)$ are either all bounded or all unbounded, by the
previous argument in the primary case, since the underlying lattice is
the same. Therefore, the
$A$-graded Hilbert function of $\kk[\NN^n]/I$ is bounded above if and
only if the $A$-graded Hilbert function of
$\kk[\NN^n]/(I+I_1)$ is bounded above. Noting that the rings 
$\kk[\NN^n]/(I+I_j)$, $j=1,\dots,g$, have the same dimension, which 
thus equals $\dim(\kk[\NN^n]/I)$, the proof of the desired
equivalences is reduced to the primary case.
\end{proof}

\begin{definition}
\label{d:toral}
Let $I \subset \kk[\NN^n]$ be an $A$-homogeneous mesoprimary ideal.
We say that $\kk[\NN^n]/I$ (or $I$ itself) is \emph{toral} if one of the equivalent conditions of
Lemma~\ref{l:toralConditions} is satisfied. Otherwise, $\kk[\NN^n]/I$
and $I$ are called \emph{Andean}.  
Note that both of these properties depend on the $A$-grading.  
\end{definition}

\begin{remark}
\label{r:toral}
The name ``Andean'' is a pictorial description of the grading of
quotients by Andean ideals. If $I \subseteq \kk[\NN^n]$ is an Andean
prime, the set
\[
\{ \beta \in \ZZ^d \mid (\kk[\NN^n]/I)_\beta \neq 0 \}
\]
consists of the lattice points on a translate of a face of the cone
$\RR_{>0}A$ (not necessarily a proper face). Since the Hilbert function is unbounded, the picture of
a very high, long and thin mountain range comes to mind.
See also~\cite[Remark~5.3]{dmm2}.
\endrk
\end{remark}

\begin{example}[{\cite[Example~1.7]{dmm2}}]
Consider $I = \<xz - y,xw - y\> = \<z - w, xw - y\> \cap \<x, y\>$, graded such that $\deg(x) = (1,0)$, $\deg(y) = (1,1)$, and $\deg(z) = \deg(w) = (0,1)$.  We claim the~first component is toral and the second is Andean.

Indeed, $\kk[x,y,z,w]/\<z - w, xw - y\>$ has Hilbert function 1 in degree $(a,b) \in \ZZ_{\ge 0}^2$.  On the other hand, the Hilbert function of $\kk[x,y,z,w]/\<x, y\> = \kk[z,w]$ is 0 in degree $(a,b)$ whenever $a > 0$, while in degree $(0,b)$ with $b \geq 0$, the Hilbert function is $b+1$, which is unbounded.  

To make this more interesting, one can consider the Hilbert function of $\kk[x,y,z,w]/I$.  In this case, the Hilbert function is 1 in degree $(a,b)$ when $a$ is positive, and $b+1$ when $a = 0$. 
\end{example}

\begin{corollary}
\label{c:toralConditions}
Each prime associated to a toral mesoprimary ideal is toral, and every
prime associated to an Andean mesoprimary ideal is Andean.   
\end{corollary}

If $I$ is mesoprimary, then $I$ is either toral or Andean. We note that $I$ is
toral if and only if $\kk[\NN^n]/I$ is a toral module in the sense
of~\cite[Definition~4.1]{dmm2}; and $I$ is Andean if and only if
$\kk[\NN^n]/I$ is an Andean module in the sense of~\cite[Definition~5.1]{dmm2}.

\begin{lemma}
\label{lemma:embeddedToral}
Suppose $I \subseteq J$ are $A$-homogeneous mesoprimary ideals. If $I$ is toral, then so is $J$. 
\end{lemma}

\begin{proof}
Note that the $A$-graded Hilbert function of $\kk[\NN^n]/J$ is bounded
above by the $A$-graded Hilbert function of $\kk[\NN^n]/I$. If the
latter is bounded, then so is the former.
\end{proof}

A binomial ideal may have both Andean and toral minimal and embedded
primes, and the minimal prime corresponding to a toral embedded prime
may be Andean. However, any embedded prime corresponding to a toral
minimal prime must be toral. See the examples below.

On the other hand, 
whenever a cellular
$A$-homogeneous binomial ideal $I$ has at least one Andean component, then
all of the toral primes must be embedded.  Indeed, the minimal primes
of $I$ correspond to the minimal primes of the lattice ideal $I \cap
\kk[\NN^\sigma]$, and therefore, once this is Andean, all the minimal
primes are Andean, and any remaining components (including every toral
component) must be embedded.

\begin{example}
\label{e:cellularembeddedtoral}
For cellular binomial ideals, toral primes may be embedded in Andean
primes, but not the other way around.  For example, the cellular ideal
$I \subset \kk[a,b,c,d,x,y]$ given by 
$$
I = \<ad - bc, x(ac - b^2), x(bd - c^2), y(ac - b^2), y(bd - c^2), x^2, xy, y^2\>
$$
is positively graded via 
$\deg(a) = \left[ \begin{smallmatrix} 1 \\ 0 \end{smallmatrix}\right]$,
$\deg(b) = \left[ \begin{smallmatrix} 1 \\ 1 \end{smallmatrix}\right]$,
$\deg(c) = \left[ \begin{smallmatrix} 1 \\ 2 \end{smallmatrix}\right]$,
$\deg(d) = \left[ \begin{smallmatrix} 1 \\ 3 \end{smallmatrix}\right]$,
$\deg(x) = \left[ \begin{smallmatrix} 1 \\ 4 \end{smallmatrix}\right]$,
and
$\deg(y) = \left[ \begin{smallmatrix} 1 \\ 5 \end{smallmatrix}\right]$.

The ideal $I$ has three coprincipal components in the decomposition
from Theorem~\ref{t:kmessential}, with essential witnesses $1$, $x$,
and $y$.  The first yields an Andean component, and the remaining two
components have the same associated (toral) mesoprime with different
Artinian parts.  In particular, 
$$
\begin{array}{rcl}
I
&=& \<ad - bc, x,y\> \cap \<ad - bc, ac - b^2, bd - c^2, x^2, y\> \cap \<ad - bc, ac - b^2, bd - c^2, x, y^2\>\\
&=& \<ad - bc, x,y\> \cap \<ad - bc, ac - b^2, bd - c^2, x^2, xy,y^2\> \\
\end{array}
$$
upon combining the last two coprincipal components to a single
mesoprimary component.  See Figure~\ref{f:cellularembeddedtoral} for
picture of the nilpotent monomials of $I$.  
\endrk
\end{example}

\begin{figure}
\begin{center}
\includegraphics[width=1.2in]{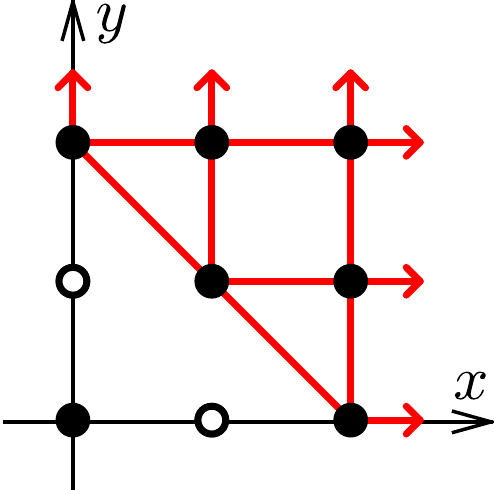}
\end{center}
\caption{Depicted above are the nilpotent monomial equivalence classes
  modulo the cellular ideal $I$ in
  Example~\ref{e:cellularembeddedtoral}.  The origin is a witness for
  a principal lattice ideal, and the monomials $x$ and $y$ are both
  witnesses for the twisted cubic.} 
\label{f:cellularembeddedtoral}
\end{figure}

\section{Some combinatorial savings when computing toral components}
\label{sec:setTo1}

An important result in~\cite{dmm} is that toral primary components of
$A$-homogeneous binomial ideals are easier to compute than Andean
ones. This statement is~\cite[Theorem~4.13]{dmm}, which contains a
minor error (see Remark~\ref{rmk:dmmErrorSetTo1}). 

We first recall how primary components are computed in~\cite{dmm}. 
Suppose that $\sigma \subseteq [n]$, $I(\rho)$ is a prime lattice ideal
in $\kk[\NN^\sigma]$, and $P=I(\rho)+\mm_{\sigma^c}$
is a toral associated prime of an $A$-homogeneous binomial ideal $I$. Then 
\cite[Theorem~3.2]{dmm} states that
in the case that $\kk$ is an algebraically closed field of
characteristic zero, the
$P$-primary component of $I$ may be chosen of the form
\begin{equation}
\label{eqn:primDec}
\big( (I+I(\rho)+K): (\prod_{i\in \sigma}x_i)^\infty \big) + M
\end{equation}
where $K, M \subseteq \kk[\NN^{\sigma^c}]$, $K$ is generated by sufficiently
high powers of the variables $x_j$, $j\notin \sigma$, and $M$ is a
monomial ideal computed combinatorially. If $P$ is a minimal prime of
$I$, then we may choose $K = \<0\>$.

We remark that the ideal $M$ above does not necessarily contain all monomials belonging to the
corresponding primary component. Even when $P$ is minimal, $\big( (I+I(\rho)): (\prod_{i\in \sigma}x_i)^\infty \big)$
may contain monomials in $\kk[\NN^{\sigma^c}]$ that belong neither to $I$ nor to $M$.

\begin{example}
\label{ex:MisNotEnough}
Let $I = \< z^2 - w^2, x(z - w), x^2 \> \subset \kk[x,z,w]$, with the
usual $\ZZ$-grading on the polynomial ring. Then the primary component
associated to the minimal prime $\<z+w,x\>$ is $\<z+w,x\>$. In this case, the monomial
ideal $M$ from~\eqref{eqn:primDec} is $M=\<x^2\>$; the monomial $x$
comes from performing $((I+\<z+w\>):(zw)^\infty)$.
\endrk
\end{example}

The monomial ideal $M$ from~\eqref{eqn:primDec} is computed by considering a congruence on
the monoid $\ZZ^\sigma\times \NN^{\sigma^c}$. The gist
of~\cite[Theorem~4.13]{dmm} is that, for toral primes, the computation 
of the monomial ideal $M$ can be performed by considering a congruence
on the (much smaller) monoid $\NN^{\sigma^c}$. This leads to
significant combinatorial savings when computing toral primary
components. 

\begin{theorem}
\label{thm:correctThm4.13}
Let $I$ be an $A$-homogeneous binomial ideal in $\kk[\NN^n]$, where $\kk$
is an algebraically closed field of characteristic zero. Let $\sigma
\subseteq [n]$, $I(\rho) \subseteq \kk[\NN^\sigma]$ a prime lattice
ideal corresponding to a character $\rho:L\to \kk^*$, and assume that
$P=I(\rho)+\mm_{\sigma^c}$ is a toral associated prime of $I$. Let
$\nu = (\nu_j)_{j \in \sigma} \in (\kk^*)^\sigma$ be a zero of $I \cap \kk[\NN^\sigma]$,
and set $\overline{I} = I\cdot \kk[\NN^n]/\<x_j - \nu_j \mid j \in
\sigma\>$. We consider $\overline{I}$ as an ideal in
$\kk[\NN^\sigma]$. 
Then a valid choice for the $P$-primary component
of $I$ is 
\begin{equation}
\label{eqn:toralComponent}
\big( (I+I(\rho)+K): (\prod_{i\in \sigma}x_i)^\infty \big) + \overline{M},
\end{equation}
where $K$ is an ideal generated by sufficiently high powers of the
variables indexed by $\sigma^c$, and $\overline{M}$ is the monomial ideal
combinatorially produced by~\cite[Theorem~3.2]{dmm} for the associated
prime $\mm_{\sigma^c}$ of $\overline{I}$.
\end{theorem}

\begin{remark}
\label{rmk:dmmErrorSetTo1}
We note that the statement of~\cite[Theorem~4.13]{dmm} contains a
minor error. Instead of setting the variables indexed 
by $\sigma$ to values given by a zero of $I \cap \kk[\NN^\sigma]$,
as we do in Theorem~\ref{thm:correctThm4.13},
those variables are set to $1$, which is a valid choice only 
when $(1)_{i\in\sigma} \in (\kk^*)^\sigma$ is a root of $I \cap
\kk[\NN^\sigma]$. If this is not the case, then setting the variables
indexed by $\sigma$ to $1$ introduces constants to $\overline{I}$.
The proof of~\cite[Theorem~4.13]{dmm} is correct, once the statement
is suitably modified.
\endrk
\end{remark}

It has been one of the goals of this project to provide 
an analogous result for computing witnesses and
coprincipal or mesoprimary components of $A$-homogeneous binomial ideals 
corresponding to toral mesoprimes. A general statement is
unfortunately out of reach.

\begin{remark}
\label{rmk:BetterToralNotPossible}
We emphasize that the
monomials of an associated mesoprime cannot necessarily be obtained by
evaluating the $\sigma$-variables.  For example, the mesoprimary
decomposition of the ideal $I = \< z^2 - w^2, x(z - w), x^2 \> \subset
\kk[x,z,w]$ constructed in Theorem~\ref{t:kmessential} is  
$$I = \< z^2 - w^2, x \> \cap \< z - w, x^2 \>,$$ 
and both components have $\<z - w, x\>$ as an associated prime.  As
such, the primary decomposition  
$$I = \< z + w, x \> \cap \< z - w, x^2 \>$$
results from taking the canonical primary decomposition of each
mesoprimary component and collecting both components with associated
prime $\<z - w, x\>$.   
\endrk
\end{remark}

While there may not be a general result along the lines of
Theorem~\ref{thm:correctThm4.13} for mesoprimary decomposition, we do
provide in Theorem~\ref{thm:whatThereIs} a special case in which lower-dimensional combinatorics can be
used for computations.

We start by introducing terminology and providing auxiliary results.

\begin{definition}
The \emph{support} of a polynomial $h$, denoted $\supp(h)$, is the set of
monomials that appear in $h$ with nonzero coefficient. 
\end{definition}

\begin{conv}
\label{conv:savings}
Until the end of this section, we use the following notation and assumptions.
Let $I$ be an $A$-homogeneous binomial ideal, where $A$ is a $d \times n$
matrix of rank $d<n$. Let $\sigma \subset [n]$, 
with $|\sigma|=d$, be such that
the matrix $A_\sigma$ consisting of the columns of $A$ indexed by
$\sigma$ has full rank $d$. We assume that $(1)_{i\in\sigma} \in
\kk^\sigma$ is a zero of $I \cap \kk[\NN^\sigma]$.
Let $\overline{I}$ be the ideal
$I \cdot (\kk[\NN^n]/\<x_i-1 \mid i\in \sigma\>)$, considered as an ideal in
$\kk[\NN^{\sigma^c}]$.
\end{conv}

\begin{proposition}
\label{prop:lifting}
Under the notation and assumptions of Convention~\ref{conv:savings},
if $g =\sum_{i=1}^r \lambda_ix^{u_i}$, where $\lambda_i \in
\kk^*$ and $u_i \in \NN^{\sigma^c}$ for $i=1,\dots,r$, 
then $g \in \overline{I}$ if and
only if there are $v_1,\dots,v_r\in \NN^{\sigma}$ such that 
$f = \sum_{i=1}^r \lambda_i x^{u_i}x^{v_i} \in I$.
\end{proposition}

\begin{proof}
We note that if $f = \sum_{i=1}^r \lambda_i x^{u_i}x^{v_i}$ as above
belongs to $I$, then $g = \sum_{i=1}^r \lambda_ix^{u_i} \in \overline{I}$.

Assume $g = \sum_{i=1}^r \lambda_ix^{u_i}\in\overline{I}$. Since
$\overline{I}$ is a binomial ideal, there are $\mu_1,\dots,\mu_s\in
\kk^*$  and binomials $\overline{b}_1,\dots,\overline{b}_s
\in \overline{I}\subset \kk[\NN^{\sigma^c}]$, where each
$\overline{b}_i$ denotes the image in $\overline{I}$ of a binomial $b_i \in I$,
such that
$g =\mu_1\overline{b}_1+\dots+ \mu_s \overline{b}_s$. 
We may assume that $b_1,\dots,b_s$ are $A$-homogeneous and no two of
them have the same support. Since
$A_\sigma$ is invertible, we see that if $\overline{b}_i$ has one
term, then $b_i$ has one term. 
Moreover, we may choose $b_1,\dots,b_s$
in such a way that if $i\neq j$, then $\overline{b}_i$ and
$\overline{b}_j$ have different supports. To see this, suppose that
$b_i = x^wx^p-\kappa x^vx^q$ and $b_j = x^{w'}x^p-\kappa'x^{v'}x^q$,
where $p,q \in \NN^{\sigma^c}$, $v,v',w,w' \in \NN^{\sigma}$ and
$\kappa,\kappa' \in \kk$. Since $b_i$ and $b_j$ are $A$-homogeneous,
and $A_\sigma$ is invertible, we have that $w-v=w'-v'$, so that
$x^{w'}b_i$ and $x^wb_j$ have the same support. If $x^{w'}b_i=x^wb_j$,
we may use this binomial instead of $b_i$ and $b_j$. If
$x^{w'}b_i\neq x^wb_j$, then $x^{w+w'}x^p, x^{v+v'}x^q \in I$, and we
may use these monomials instead of $b_i$ and $b_j$.

If the binomials $\overline{b}_1,\dots,\overline{b}_s$ 
have pairwise disjoint supports, then $f =\sum_{i=1}^s \mu_ib_i \in
I$ satisfies the required conditions.

Now suppose that $x^m\in \kk[\NN^{\sigma^c}]$ belongs to the support
of at least two of the binomials $\overline{b}_i$. 
For each $i$ such that $x^m$ is in the support of 
$\overline{b}_i$, there exists
$x^{q_i} \in \kk[\NN^{\sigma}]$ such that $x^mx^{q_i}$ is in the
support of $b_i$. Let $x^q \in \kk[\NN^{\sigma}]$ be the
least common multiple of all such $x^{q_i}$. Then the coefficient of $x^m$
in the sum of the $\mu_i\overline{b}_i$ over all $\overline{b}_i$
containing $x^m$ in their support equals the coefficient of
$x^qx^m$ in the sum of the $\mu_i x^{q-q_i}b_i$ over all $b_i$ containing a
multiple of $x^m$ in their support.

If $x^m$ is the only monomial appearing the support of more than one
$\overline{b_i}$, then the polynomial $f$ constructed as the sum 
over $\overline{b}_i$ containing $x^m$ of $\mu_i x^{q-q_i}b_i$ plus
the sum over $\overline{b}_i$ not containing $x^m$ of $\mu_i b_i$
satisfies the required conditions. 

If there exists $x^{m'} \in \kk[\NN^{\sigma^c}]$, $x^{m'} \neq x^m$, 
appearing in the support of more than one of the $\overline{b}_j$, 
and $x^m$, $x^{m'}$ are the only two monomials with this property,
we
repeat the same procedure as before, obtaining binomials $\mu_j x^{q'-q'_j}b_j$,
with the proviso that if the support of $\overline{b}_\ell$ equals
$\{x^m,x^{m'}\}$, then we use 
\[
f=\mu_\ell x^{q+q'-q_\ell-q'_\ell}b_\ell +
\sum_{\substack{i\neq \ell  \\ x^m \in \supp(\overline{b}_i)}} \mu_i x^{q+q'-q_\ell'-q_i} b_i 
+
\sum_{\substack{j \neq \ell \\ x^{m'} \in \supp(\overline{b}_j)}} \mu_j
  x^{q+q'-q_\ell-q'_j}b_j
+
\sum_{x^m,x^{m'} \notin \supp(\overline{b}_k)} \mu_k b_k.
\]

If there is $x^{m''} \in \kk[\NN^{\sigma^c}]$, different from $x^m$ and
$x^{m'}$ that appears in more than one support, we repeat the
procedure, taking care that if there is a binomial $\overline{b}_i$ with support
$\{x^m, x^{m''}\}$ and/or a binomial $\overline{b}_j$ with support $\{
x^{m'}, x^{m''}\}$, then all of the binomials $b_k$ involving multiples
of $x^m$ or $x^{m'}$ need to be multiplied by additional monomials in
$\kk[\NN^\sigma]$.

Continuing in this manner, we obtain the desired $f$.
\end{proof}

\begin{proposition}
\label{prop:nonLifting}
Under the assumptions and notation of Convention~\ref{conv:savings},
if $p \in \kk[\NN^n]$ is an $A$-homogeneous polynomial
not belonging to $I_\sigma=\big(I : (\prod_{i\in \sigma} x_i)^\infty
\big)$, then its image $\overline{p}$ in $\kk[\NN^{\sigma^c}]$ under
the map that sets to $1$ the variables indexed by $\sigma$ does not belong
to $\overline{I}$. 
\end{proposition}

\begin{proof}
We prove the contrapositive statement by 
induction on the cardinality of the support of $p$. If
$p=\lambda x^ux^m$, where $\lambda \in \kk^*, x^u \in
\kk[\NN^{\sigma}], x^m \in \kk[\NN^{\sigma^c}]$, and $\overline{p}=
\lambda x^m \in \overline{I}$, apply Proposition~\ref{prop:lifting} to
obtain a monomial $x^v\in \kk[\NN^{\sigma}]$ such that $\lambda x^vx^m
\in I$. This implies $x^m\in I_\sigma$, and therefore $p \in
I_\sigma$. 

Now let $p = \sum_{i=1}^r \lambda_i x^{u_i} x^{m_i} \in \kk[\NN^n]$, 
where $\lambda_i\in \kk^*, x^{u_i}\in \kk[\NN^\sigma], x^{m_i} \in
\kk[\NN^{\sigma^c}]$ for $i=1,\dots,r$ and 
$r \geq 2$.
Assume that $\overline{p} \in \overline{I}$. By
Proposition~\ref{prop:lifting}, there are $x^{v_1},\dots,x^{v_r} \in
\kk[\NN^\sigma]$ such that $p' = \sum \lambda_i x^{v_i} x^{m_i} \in
I$. Since $I$ is $A$-homogeneous, the homogeneous components of $p'$ belong
to $I$. Denote by $q$ one such component. As the matrix $A_\sigma$ is
invertible, we can find monomials $x^u, x^v \in \kk[\NN^\sigma]$ such
that $x^u p$ and $x^v q$ have the same $A$-degree. Now, let $i$ be such
that $\lambda_i x^{v_i} x^{m_i}$ is a monomial in $q$. Then $x^u
x^{u_i} x^{m_i}$ and $x^v x^{v_i} x^{m_i}$ have the same
$A$-degree. Using again the fact that $A_\sigma$ is invertible, we see
that $u+u_i = v+v_i$. This implies that the support of $x^u p-x^vq$ is
strictly contained in the support of $p$, and since $q \in I$,
the image $\overline{q}$ of $x^vq$ under setting to $1$ the variables
indexed by $\sigma$ belongs to $\overline{I}$. By induction, $x^up-x^vq \in I_{\sigma}$
Moreover, since $q\in I
\subset I_\sigma$, we see that $x^up \in I_\sigma$, and by definition
of $I_\sigma$, this implies that $p \in I_\sigma$.
\end{proof}

\begin{definition}
\label{d:weakWitness}
Fix $\sigma \subset [n]$ and a
binomial ideal $I \subset \kk[\NN^n]$. Set
$I_\sigma=\big(I:(\prod_{i\in \sigma}x_i)^\infty\big)$.
\begin{enumerate}[(a)]
\item 
A \emph{weak monomial $I$-witness for
$\mm_{\sigma^c}$} is a monomial $x^w \in \kk[\NN^{\sigma^c}]$, $x^w
\notin I$, such that there exists
$x^m \in \kk[\NN^\sigma]$ with the property that for each $i \notin
\sigma$, there are
a monomial $x^{q_i}$ and a scalar $\lambda_i \in \kk^*$ such
that $x_i (x^m x^w - \lambda_i x^{q_i}) \in I_\sigma$ but $x^mx^w
-\lambda_i x^{q_i} \notin I_\sigma$.  
\item 
A weak monomial $I$-witness $x^w$ for $\mm_{\sigma^c}$ is
\emph{essential} if there is a
polynomial $p \in \kk[\NN^n]$,
$p\notin I_\sigma$, and a
monomial $x^v \in \kk[\NN^\sigma]$ such that $x^vx^w$ is a monomial in
$p$, and
$x_j p \in I_\sigma$ for
all $j \notin \sigma$.
\end{enumerate}
\end{definition}

The difference between Definition~\ref{d:witness} and
Definition~\ref{d:weakWitness} is that the ideal $I$ is not assumed to
be positively graded. This is a profound difference, as weak monomial witnesses
are \textbf{not} monomial witnesses in the sense of~\cite{kmmeso},
because the conditions on divisibility imposed by the definitions in~\cite{kmmeso}
are not (necessarily) satisfied.

\begin{theorem}
\label{thm:whatThereIs}
Under the assumptions and notation of Convention~\ref{conv:savings}, 
a monomial $x^w \in \kk[\NN^{\sigma^c}]$ is an (essential)
monomial $I$-witness for $\mm_{\sigma^c} \subset \kk[\NN^n]$ if and
only if it is a weak
(essential) monomial $\overline{I}$-witness for $\mm_{\sigma^c}
\subset \kk[\NN^{\sigma^c}]$.
\end{theorem}

\begin{proof}
Let $x^w \in \kk[\NN^{\sigma^c}]$ be a monomial $I$-witness for
$\mm_{\sigma^c} \subset \kk[\NN^n]$. Using
Proposition~\ref{prop:nonLifting}, we see that
the images under setting to $1$ the variables indexed by $\sigma$ of
the auxiliary binomials required for $x^w$ to be a monomial $I$-witness,
satisfy the conditions required for $x^w$ to be a weak monomial
$\overline{I}$-witness. If $x^w$ is an essential monomial $I$-witness,
the image of the auxiliary polynomial $p$ serves to verify that $x^w$
is an essential weak monomial $\overline{I}$-witness.

Now assume that $x^w$ is a weak monomial $\overline{I}$-witness for
$\mm_{\sigma^c} \subset \kk[\NN^{\sigma^c}]$, and
for each $i \in \sigma^c$, let $\lambda_i \in \kk^*$ and $x^{q_i} \in
\kk[\NN^\sigma]$ such that $x_i(x^w - \lambda_i x^{q_i}) \in
\overline{I}$ but $x^w -\lambda_i x^{q_i} \notin \overline{I}$. By
Propositions~\ref{prop:lifting} and~\ref{prop:nonLifting} there are
monomials $x^{u_i}, x^{v_i} \in \kk[\NN^\sigma]$ such that
$x_i(x^{u_i}x^w-\lambda_i x^{v_i}x^{q_i}) \in I \subset I_\sigma$ and
$(x^{u_i}x^w-\lambda_i x^{v_i}x^{q_i}) \notin I_\sigma$. Taking $x^u$
to be the least common multiple of the $x^{u_i}$, we see that the
binomials $x^u x^w - \lambda_i x^{u-u_i} x^{v_i} x^{q_i}$ satisfy the
properties necessary to ensure that $x^w$ is a monomial $I$-witness
for $\mm_{\sigma^c} \subset \kk[\NN^n]$.

Finally, if $x^m$ is a weak essential monomial $\overline{I}$-witness
for $\mm_{\sigma^c} \subset \kk[\NN^{\sigma^c}]$,
then in particular $x^w$ is a weak monomial $\overline{I}$-witness for
$\mm_{\sigma^c} \subset \kk[\NN^{\sigma^c}]$, and by the previous
argument, it is a monomial $I$-witness for $\mm_{\sigma^c} \subset \kk[\NN^n]$.
Now let $g=\sum_{i=1}^r \mu_i x^{m_i} \in \kk[\NN^{\sigma^c}]$,
$\mu_1,\dots,\mu_r \in \kk^*$, such that
$g \notin \overline{I}$, $m_1 = w$, and $x_ig
\in \overline{I}$ for all $i \in \sigma^c$. Fix $i_0 \in \sigma^c$. By
Propositions~\ref{prop:lifting} and~\ref{prop:nonLifting} applied to
$x_{i_0}g$, there are
monomials $x^{v_1},\dots,x^{v_r} \in \kk[\NN^{\sigma}]$ such that 
$p' = \sum_{i=1}^r \mu_i x^{v_i}x^{m_i} \notin I_\sigma$ and $x_{i_0} p'
\in I \subset I_\sigma$. By Proposition~\ref{prop:nonLifting}, if $x_j
\in \sigma^c$, $j \neq i_0$, the fact that $x_jg \in \overline{I}$ implies that $x_j
p' \in I_\sigma$. Now let $p$ be the $A$-homogeneous component of $p'$
containing the monomial $x^{v_1}x^{m_1} = x^{v_1} x^w$. Since $I$ and
$I_\sigma$ are $A$-homogeneous, the polynomial $p$ satisfies the conditions
necessary to ensure that $x^m$ is an essential monomial $I$-witness
for $\mm_{\sigma^c} \subset \kk[\NN^n]$.
\end{proof}

\section{The toral part of a binomial primary decomposition}
\label{sec:Itoral}

As~\cite[Proposition~6.4]{dmm2} shows, it sometimes makes sense to discard the Andean
components of an $A$-homogeneous binomial ideal. The goal of this section
is to show that this process may not result in a binomial ideal (Example~\ref{e:nonbinomialItoral}).

\begin{definition}
\label{d:Itoral}
Fix an $A$-homogeneous binomial ideal $I \subset \kk[\NN^n]$, where $\kk$
is algebraically closed. Let $I = \cap_{\ell=1}^r J_{\ell}$
and a binomial primary
decomposition, where $J_1,\dots,J_t$ are toral and $J_{t+1},\dots,J_r$
are Andean.
The \emph{toral part} of (this decomposition of) $I$, denoted
$I_\toral$, equals the intersection $\cap_{\ell}^{t} J_\ell$ of the toral 
components (cf.~\cite[Proposition~6.4]{dmm2}). 
\end{definition}

Since embedded primary components are not uniquely determined, the
ideal $I_\toral$ in Definition~\ref{d:Itoral} depends on the primary decomposition
unless all the toral associated primes of $I$ are
minimal.

\begin{lemma}
Let $I$ be an $A$-homogeneous binomial ideal in $\kk[\NN^n]$, and let
$I=\cap_{\ell=1}^p I_\ell$ be a mesoprimary decomposition of $I$.
Assume that $I_1,\dots,I_q$ are toral and $I_{q+1},\dots,I_p$ are
Andean, and consider $J=\cap_{\ell=1}^q I_\ell$. Then $J$ equals
$I_{\textnormal{toral}}$ for the primary decomposition of $I$ obtained
by primary decomposing the mesoprimary components $I_\ell$. 
\end{lemma}

\begin{proof}
The reason this is not immediate is that, when the mesoprimary
components $I_\ell$ are primary decomposed, some of the resulting
primary ideals may not be components of $I$. The question of how to
eliminate possible redundancies in this process is a subtle one
\cite[Remark~16.11]{kmmeso}; however, in this case, we need only
observe that cancellations cannot occur between Andean and toral
mesoprimary components, as the corresponding collections of associated
primes are disjoint.
\end{proof}

\begin{example}
\label{e:nonbinomialItoral}
It is possible for the toral part of a binomial primary decomposition
to not be a binomial ideal, even when all the toral associated primes
are minimal.
Let 
$$I=\<x_3^4x_5-x_4^3x_6^3,x_3^5x_6-x_4^5x_5^2,x_1x_5^3-x_2^4x_6,x_1^4x_6^5-x_2^2x_5^2\> \subseteq \kk[x_1,\dots,x_6].$$
The ideal $I$ is (positively) 
$A$-homogeneous, for the matrix $A=\left[ \begin{smallmatrix} 5 & 5 & -11 & -13 & 5 & 0 \\ 60 & 73 &
  -130 & -160 & 82 & 14 \end{smallmatrix}\right]$
and has seven associated
primes. Five of these are toral, and all those are minimal, which
means that their corresponding primary 
components are uniquely determined. Consequently,
$I_{\textnormal{toral}}$ is independent of the primary
decomposition of $I$. In this example, it can be verified using the
\texttt{Macaulay2} package \texttt{Binomials} that
$I_{\textnormal{toral}}$ is not binomial.
\endrk
\end{example}

\begin{remark}
\label{rmk:dmmWrongAgain}
While working on this article, we found a small error
in~\cite[Example~4.10]{dmm}.

In~\cite[Example~4.10]{dmm}, it is incorrectly claimed that the dimension of any
associated prime of a lattice basis ideal $I$ 
is at least the rank $d$ of its grading matrix. This error leads
to the false conclusion that all 
toral associated primes of a lattice basis ideal $I$ have dimension exactly $d$ and are
therefore minimal. Consequently, the 
description in~\cite{dmm} does not capture all
of the toral associated primes of a
lattice basis ideal (or even all of the minimal ones).

Moreover, toral primes of a lattice basis ideal may be
embedded. (For example, $\<x_1,x_2,x_4\>$ is a toral embedded
prime of the lattice basis ideal
$$I=\<x_1^3-x_2^2x_3,x_1^2-x_2x_4^5,x_1^2-x_2^2x_5^3\>\subset
\kk[x_1,\dots,x_5]$$
with grading matrix 
$A=
\left[ \begin{smallmatrix} \phantom{-}5 & \phantom{-}5 & 5& 1 & 0 \\
-3 & -6 & 3 & 0 & 2 \end{smallmatrix} \right]$
.)
For this reason, a complete 
description of the toral associated primes of a lattice basis ideal is
not feasible using only the combinatorics of its defining matrix
matrix, since, as~\cite[Example~3.1]{HS} shows, it is not possible to
determine the embedded primes of a lattice basis ideals from the sign
patterns of the entries of the underlying matrix.

We remark that this error in~\cite{dmm} is carried over
to~\cite[Section~7]{dmm2}, although the only false result there is~\cite[Lemma~7.2]{dmm2}.
This affects the statements of~\cite[Lemma~7.4, Proposition~7.6, Theorem
7.14, Theorem~7.18, Corollary~7.25]{dmm2}, in which a certain matrix $M$ is
assumed to be square invertible, an assumption that comes from the
incorrect~\cite[Lemma~7.2]{dmm2}. Fortunately, none of these results
actually need the assumptions on $M$, and the only modification needed
is in the verification that the second display of the proof of 
of~\cite[Theorem~7.14]{dmm2} is valid. We also point out that the display
in~\cite[Example~3.7]{dmm2} should read $\mathscr{C}_{\rho,J} =
((I(B)+I_\rho):\del_J^\infty)+U_M$. 
\endrk
\end{remark}


\end{document}